\documentclass[11pt]{amsart}
\usepackage{amsmath}
\newtheorem{thm}{Theorem}[section]
\newtheorem{cor}[thm]{Corollary}
\newtheorem{lem}[thm]{Lemma}
\newtheorem{prop}[thm]{Proposition}
\theoremstyle{definition}

\theoremstyle{remark}

\numberwithin{equation}{section}

\newcommand{\C}{\mathbb{C}}

\begin{document}
\title[J-class of  abelian group of afiine maps  on $\mathbb{C}^{n}$ and Hypercyclicity]
{J-class of finitely generated abelian semigroups of affine mpas on $\mathbb{C}^{n}$ and Hypercyclicity}

\author{Yahya N'dao}

 \address{Yahya N'dao, University of Moncton, Department of mathematics and statistics, Canada}
 \email{yahiandao@yahoo.fr ; yahiandao@voila.fr}

\subjclass[2000]{47A16}

\keywords{Hypercyclic, affine maps, orbit, locally hypercyclic,
J-class, dense orbit, semigroup, abelian, subgroup} \maketitle

\begin{abstract} We give a characterization of hypercyclic using (locally hypercyclic)
 of semigroup $\mathcal{G}$ of affine maps of $\mathbb{C}^{n}$. We
 prove the existence of a $\mathcal{G}$-invariant open subset of
 $\mathbb{C}^{n}$ in which any locally hypercyclic orbit is
 dense in $\mathbb{C}^{n}$.
\end{abstract}
\maketitle

\section{\bf Introduction }
Let $M_{n}(\mathbb{C})$ be the set of all square matrices of order
$n\geq 1$  with entries in  $\mathbb{C}$ and $GL(n, \ \mathbb{C})$
be the group of all invertible matrices of $M_{n}(\mathbb{C})$. A
map $f: \ \mathbb{C}^{n}\longrightarrow \mathbb{C}^{n}$  is called
an affine map if there exist $A\in M_{n}(\mathbb{C})$  and  $a\in
\mathbb{C}^{n}$  such that  $f(x)= Ax+a$,  $x\in \mathbb{C}^{n}$.
We denote  $f= (A,a)$, we call  $A$ the \textit{linear part} of
$f$. The map $f$ is invertible if $A\in GL(n, \mathbb{C})$. Denote
by $MA(n, \ \mathbb{C})$  the vector space of all affine maps on
$\mathbb{C}^{n}$ and $GA(n, \ \mathbb{C})$ the group of all
invertible affine maps of $MA(n, \mathbb{C})$.

Let  $\mathcal{G}$  be an abelian affine sub-semigroup of $MA(n, \
\mathbb{C})$. For a vector  $v\in \mathbb{C}^{n}$,  we consider
the orbit of  $\mathcal{G}$  through  $v$: $ \mathcal{G}(v) =
\{f(v): \ f\in \mathcal{G}\} \subset \mathbb{\mathbb{C}}^{n}$.
Denote by $\overline{E}$
 the closure of a subset $E\subset \mathbb{C}^{n}$.
The semigroup $\mathcal{G}$ is called \textit{hypercyclic} if
there exists a vector $v\in {\mathbb{C}}^{n}$ such that
$\overline{\mathcal{G}(v)}={\mathbb{C}}^{n}$. For an account of
results and bibliography on hypercyclicity, we refer to the book
\cite{bm} by Bayart and Matheron.
\
We refer the reader to the recent book \cite{bm} and \cite{p} for
a thorough account on hypercyclicity. Costakis and Manoussos in
\cite{chm2} ``localize`` the concept of hypercyclicity using
J-sets. By analogy, we generalize this notion to affine case as
follow:  Suppose that $\mathcal{G}$ is generated by $p$ affines
maps $f_{1},\dots,f_{p}$ $(p\geq 1)$ then for $x\in
\mathbb{C}^{n}$, we define the extended limit set
J$_{\mathcal{G}}(x)$ to be the set of $y\in \mathbb{C}^{n}$ for
which there exists a sequence of vectors $(x_{m})_{m}$ with
$x_{m}\longrightarrow x$ and sequences of non-negative integers
$\{k^{(j)}_{m}:\ m\in \mathbb{N}\}$ for $j = 1,2,\dots, p$ with $$
(1.1)\ \ \ \ \ \ k^{(1)}_{m} + k^{(2)}_{m}+\dots+
k^{(p)}_{m}\longrightarrow+\infty$$ such that
$$f_{1}^{k^{(1)}_{m}}f_{2}^{k^{(2)}_{m}}\dots
f_{p}^{k^{(p)}_{m}}x_{m}\longrightarrow y.$$

Note that condition (1.1) is equivalent to having at least one of
the sequences $\{k^{(j)}_{m}:\ m\in\mathbb{N}\}$  for $j =
1,2,\dots, p$ containing a strictly increasing subsequence tending
to $+\infty$. We say that $\mathcal{G}$ is \textit{locally
hypercyclic} if there exists a vector
 $v\in {\mathbb{C}}^{n}\backslash\{0\}$ such that
  J$_{\mathcal{G}}(v) = {\mathbb{C}}^{n}$.
So, the question to investigate is the following: When an abelian
sub-semigroup of $MA(n, \mathbb{C})$ can be hypercyclic?

The main purpose of this paper is twofold: firstly, we give a
general characterization of the above question for any abelian
\textit{sub-semigroup} of $MA(n, \mathbb{C})$ using J-sets.
Secondly, we generalize the results proved  in \cite{JCL}, by
A.Ayadi and H.Marzougi, for linear semigroups, which  answers
negatively, the question raised in the paper of Costakis and
Manoussos \cite{chm2}: is it true that a locally hypercyclic
abelian semigroup $H$ generated by matrices $A_{1},\dots,A_{n}$ is
hypercyclic whenever J$_{H}(x)=\mathbb{C}^{n}$ for a finite set of
$x\in \mathbb{C}^{n}$ whose vector space is equal
$\mathbb{C}^{n}$? Similarly for  $\mathbb{R}^{n}$.
\bigskip

Denote by $\mathcal{B}_{0}=(e_{1},\dots,e_{n})$ the canonical
basis of $\mathbb{C}^{n}$. Let $P\in \textrm{GL}(n, \mathbb{C})$.

Let introduce the following notations and definitions. Denote by:
\
\\
 Let $n\in\mathbb{N}_{0}$  be fixed, denote by:
\\
\textbullet \ $\mathbb{C}^{*}= \mathbb{C}\backslash\{0\}$,
$\mathbb{R}^{*}= \mathbb{R}\backslash\{0\}$  and $\mathbb{N}_{0}=
\mathbb{N}\backslash\{0\}$.
\
\\
\textbullet \ $\mathcal{B}_{0} = (e_{1},\dots,e_{n+1})$ the
canonical basis of $\mathbb{C}^{n+1}$  and   $I_{n+1}$  the
identity matrix of $GL(n+1,\mathbb{C})$.
\
\\
For each $m=1,2,\dots, n+1$, denote by:
\\
\textbullet \ $\mathbb{T}_{m}(\mathbb{C})$ the set of matrices
over $\mathbb{C}$ of the form \begin{align}\begin{bmatrix}
  \mu & \ &  \ & 0 \\
  a_{2,1} & \mu &  \ &  \\\
  \vdots &  \ddots & \ddots & \ \\
  a_{m,1} & \dots & a_{m,m-1} & \mu
\end{bmatrix}& \  \label{eq1}
\end{align}
\\
\textbullet \ $\mathbb{T}_{m}^{\ast}(\mathbb{C})$  the group of
matrices of the form (~\ref{eq1}) with $\mu\neq 0$.
\
\\
Let  $r\in \mathbb{N}$ and $\eta
=(n_{1},\dots,n_{r})\in\mathbb{N}^{r}_{0}$ such that $n_{1}+\dots
+ n_{r}=n+1. $ In particular, $r\leq n+1$. Write
\
\\
\textbullet \ $\mathcal{K}_{\eta,r}(\mathbb{C}): =
\mathbb{T}_{n_{1}}(\mathbb{C})\oplus\dots \oplus
\mathbb{T}_{n_{r}}(\mathbb{C}).$ In particular if  $r=1$, then
$\mathcal{K}_{\eta,1}(\mathbb{C}) = \mathbb{T}_{n+1}(\mathbb{C})$
and  $\eta=(n+1)$.
\
\\
\textbullet \ $\mathcal{K}^{*}_{\eta,r}(\mathbb{C}): =
\mathcal{K}_{\eta,r}(\mathbb{C})\cap \textrm{GL}(n+1, \
\mathbb{C})$.\
\
\\
\textbullet \ $u_{0} = (e_{1,1},\dots,e_{r,1})\in
\mathbb{C}^{n+1}$ where
 $e_{k,1} = (1,0,\dots,
0)\in \mathbb{C}^{n_{k}}$, for $k=1,\dots, r$. So
$u_{0}\in\{1\}\times\mathbb{C}^{n}$.
\\
\textbullet \ $p_{2}:\mathbb{C}\times
\mathbb{C}^{n}\longrightarrow\mathbb{C}^{n}$ the second projection
defined by $p_{2}(x_{1},\dots,x_{n+1})=(x_{2},\dots,x_{n+1})$.\
\\
 \textbullet \ Define the map \  \  \ $\Phi\ :\ GA(n,\ \mathbb{C})\
\longrightarrow\
 GL(n+1,\ \mathbb{C})$ \\
  $$f =(A,a) \ \longmapsto\ \begin{bmatrix}
                     1  & 0 \\
                     a & A
                 \end{bmatrix}$$
We have the following composition formula
 $$\begin{bmatrix}
                     1  & 0 \\
                     a & A
                 \end{bmatrix}\begin{bmatrix}
                     1  & 0 \\
                     b & B
                 \end{bmatrix} = \begin{bmatrix}
                     1  & 0 \\
                     Ab+a & AB
                 \end{bmatrix}.$$
 Then  $\Phi$  is an injective homomorphism of groups.\ Write
  \
\\
\textbullet \; $G=\Phi(\mathcal{G})$, it is an abelian subgroup of
$GL(n+1, \mathbb{C})$.
\\

Let consider the normal form of $G$: By Proposition ~\ref{p:2},
there exists a $P\in \Phi(\textrm{GA}(n, \mathbb{C}))$ and a
partition $\eta$ of $(n+1)$ such that $G'=P^{-1}GP\subset
\mathcal{K}^{*}_{\eta,r}(\mathbb{C})\cap\Phi(GA(n,\mathbb{C}))$.
For such a choice of matrix $P$,  we let
\
\\
\textbullet \; $v_{0} = Pu_{0}$. So $v_{0}\in
\{1\}\times\mathbb{C}^{n}$, since
  $P\in\Phi(GA(n, \mathbb{C}))$.
 \\
 \textbullet \; $w_{0} = p_{2}(v_{0})\in\mathbb{C}^{n}$. We have $v_{0}=(1, w_{0})$.\
 \\
\textbullet \; $\varphi=\Phi^{-1}(P)\in GA(n,\mathbb{C})$.
 \\

 Denote by
$V':=\underset{k=1}{\overset{r}{\prod}}\mathbb{C}^{*}\times\mathbb{C}^{n_{k}-1}$,
$V=P(V')$ and $U'=p_{2}(V')$, then

$$(1)\ \ \ \ \ \ \ \ \ \ \   \ \ \ U'=\left\{\begin{array}{cc}
  \mathbb{C}^{n_{1}-1}\times\underset{k=2}{\overset{r}{\prod}}\mathbb{C}^{*}\times\mathbb{C}^{n_{k}-1}& \ \ if\ r\geq2 \\
  \\
  \mathbb{C}^{n_{1}-1} \ \ \ \ \ \  \ \ \ \ \ \ \ \  \ \ \ \ \ \ \ \  &\ \ if\ r=1.
\end{array}\right.$$
\medskip

 For such choise of the matrix $P\in \Phi(GA(n+1,
\C))$, we can write $$P=\left[\begin{array}{cc}
  1 & 0 \\
  d & Q
\end{array}\right],\ \ \mathrm{with}\ \ \ Q\in GL(n, \C).$$  We have $\varphi=(Q,d)\in GA(n, \C)$,
 $U=\varphi(U')$ and $G^{*}=G\cap GL(n+1,\C)$. We have $G^{*}$ is an abelian semigroup of $GL(n+1,\C)$.
 \medskip

Our principal results are the following:
\bigskip

\begin{thm}\label{t:1} Let $\mathcal{G}$ be a finitely generated abelian semigroup of affine maps on $\mathbb{C}^{n}$.
 If $J_{\mathcal{G}}(v)=\mathbb{C}^{n}$ for some $v\in U$ then
$\overline{\mathcal{G}(v)} = \mathbb{C}^{n}$.
\end{thm}
\medskip

\begin{cor}\label{c:1} Under the hypothesis of Theorem \ref{t:1}, the following are equivalent:
\begin{itemize}
\item [(i)] $\mathcal{G}$ is hypercyclic.
\item [(ii)] $J_{\mathcal{G}}(w_{0}) = \mathbb{C}^{n}$.
\item [(iii)] $\overline{\mathcal{G}(w_{0})} = \mathbb{C}^{n}$.
\end{itemize}
\end{cor}
\medskip

\begin{cor}\label{c:2} Under the hypothesis of Theorem \ref{t:1}, set $\mathcal{A}=\{x\in \mathbb{C}^{n}:  J_{\mathcal{G}}(x)= \mathbb{C}^{n}\}$.
If $\mathcal{G}$ is not hypercyclic then $\mathcal{A}\subset
\underset{k=1}{\overset{r}{\bigcup}}H_{k}$, ($r\leq n$) where
$H_{k}$ are $\mathcal{G}$-invariant affine subspaces of
$\mathbb{C}^{n}$ with dimension $n-1$.
\end{cor}
\medskip

\section{\bf Preliminaries and basic notions}
 \bigskip

\begin{prop}\label{p:2} Let  $\mathcal{G}$ be an abelian sub-semigroup of
$MA(n,\mathbb{C})$ and  $G=\Phi(\mathcal{G})$. Then there exists
$P\in \Phi(GA(n,\mathbb{C}))$ such that $P^{-1}GP$ is a
sub-semigroup of $\mathcal{K}_{\eta,r}(\mathbb{C})\cap\Phi(GA(n,
\mathbb{C}))$,
 for some $r\leq n+1$ and $\eta=(n_{1},\dots,n_{r})\in\mathbb{N}_{0}^{r}$.
\end{prop}
\medskip

To prove Proposition ~\ref{p:2}, we need the following
proposition.
\medskip

\begin{prop}\label{p:002}$($\cite{AA12}, Proposition 2.1$)$ Let  $\mathcal{G}$ be an abelian subgroup of
$GA(n,\mathbb{C})$ and  $G=\Phi(\mathcal{G})$. Then there exists
$P\in \Phi(GA(n,\mathbb{C}))$ such that $P^{-1}GP$ is a subgroup
of $\mathcal{K}^{*}_{\eta,r}(\mathbb{C})\cap\Phi(GA(n,
\mathbb{C}))$,
 for some $r\leq n+1$ and $\eta=(n_{1},\dots,n_{r})\in\mathbb{N}_{0}^{r}$.
\end{prop}
\medskip

\begin{proof}[Proof of Proposition ~\ref{p:2}] Suppose first,  $G\subset \textrm{GL}(n+1, \mathbb{C})$.
 Let $\widehat{G}$ be the group
generated by $G$. Then $\widehat{G}$ is abelian and by Proposition
~\ref{p:002}, there exists a $P\in \Phi(GA(n,\mathbb{C}))$ such
that $P^{-1}\widehat{G}P$  is an abelian subgroup of
$\mathcal{K}_{\eta,r}^{*}(\mathbb{C})$, for some
$r\in\{1,\dots,n+1\}$ and $\eta\in(\mathbb{N}_{0})^{r}$. In
particular, $P^{-1}GP\subset
\mathcal{K}_{\eta,r}^{*}(\mathbb{C})$.

Suppose now, $G\subset M_{n+1}(\mathbb{C})$. For every $A\in G$,
there exists $\lambda_{A}\in\mathbb{C}$ such that
$(A-\lambda_{A}I_{n+1})\in \textrm{GL}(n+1, \mathbb{C})$ (one can
take $\lambda_{A}$ non eigenvalue of $A$). Write $\widehat{L}$ be
the group generated by $L:=\left\{A-\lambda_{A}I_{n+1}: \ A\in
G\right\}$. Then $\widehat{L}$ is an abelian subsemigroup of
$GL(n+1, \mathbb{C})$. Hence by above, there exists a $P\in
\Phi(GA(n,\mathbb{C}))$ such that $P^{-1}\widehat{L}P\subset
\mathcal{K}_{\eta,r}^{*}(\mathbb{C})$, for some
$\eta\in(\mathbb{N}_{0})^{r}$. As $$ P^{-1}LP
=\left\{P^{-1}AP-\lambda_{A}I_{n+1}: \ A\in G\right\}$$ then
$P^{-1}GP\subset \mathcal{K}_{\eta,r}(\mathbb{C})$. This proves
the proposition.
\end{proof}

\bigskip

\bigskip

Let $\widetilde{G}$ be the semigroup generated by $G$ and
$\mathbb{C}I_{n+1}=\{\lambda I_{n+1}:\  \  \ \lambda\in \mathbb{C}
\}$. Then $\widetilde{G}$ is an abelian sub-semigroup of
$M(n+1,\mathbb{C})$. By Proposition~\ref{p:2}, there exists
$P\in\Phi(GA(n, \mathbb{C}))$ such that $P^{-1}GP$ is a
sub-semigroup of $\mathcal{K}_{\eta,r}(\mathbb{C})$ for some
$r\leq n+1$ and $\eta=(n_{1},\dots,n_{r})\in\mathbb{N}_{0}^{r}$
and this also implies that $P^{-1}\widetilde{G}P$ is a
sub-semigroup of $\mathcal{K}_{\eta,r}(\mathbb{C})$. \ \\ \\
\medskip

\begin{lem}\label{LL1L:9} Let $x\in\mathbb{C}^{n}$ and $G=\Phi(\mathcal{G})$. The following are equivalent:\
\begin{itemize}
  \item [(i)] $\overline{\mathcal{G}(x)}=\mathbb{C}^{n}$.
  \item [(ii)] $\overline{G(1,x)}=\{1\}\times\mathbb{C}^{n}$.
  \item [(iii)] $\overline{\widetilde{G}(1,x)}=\mathbb{C}^{n+1}$.
\end{itemize}
\end{lem}
\medskip

\begin{proof} $(i)\Longleftrightarrow (ii):$  is obvious since $\{1\}\times\mathcal{G}(x)=G(1,x)$ by construction.
\
\\
$(iii)\Longrightarrow (ii):$  Let $y\in \mathbb{C}^{n}$ and
$(B_{m})_{m}$ be a sequence in $\widetilde{G}$ such that
$\underset{m\to +\infty}{lim}B_{m}(1,x)=(1,y)$. One can write
$B_{m}=\lambda_{m}\Phi(f_{m})$, with $f_{m}\in \mathcal{G}$ and
$\lambda_{m}\in\mathbb{C}^{*}$, thus $B_{m}(1,x)=(\lambda_{m},\ \
\lambda_{m}f_{m}(x))$, so $\underset{m\to
+\infty}{lim}\lambda_{m}=1$. Therefore, $\underset{m\to
+\infty}{lim}\Phi(f_{m})(1,x)=\underset{m\to
+\infty}{lim}\frac{1}{\lambda_{m}}B_{m}(1,x)=(1,y)$. Hence,
$(1,y)\in \overline{G(1,x)}$. \
\\
$(ii)\Longrightarrow (iii):$ Since
$\mathbb{C}^{n+1}\backslash(\{0\}\times
\mathbb{C}^{n})=\underset{\lambda\in
\mathbb{C}^{*}}{\bigcup}\lambda
\left(\{1\}\times\mathbb{C}^{n}\right)$
 and for every $\lambda\in\mathbb{C}^{*}$, $\lambda G(1, x)\subset \widetilde{G}(1,x)$, we get \begin{align*}
\mathbb{C}^{n+1} & =
\overline{\mathbb{C}^{n+1}\backslash(\{0\}\times
\mathbb{C}^{n})}\\ \ & =\overline{\underset{\lambda\in
\mathbb{C}^{*}}{\bigcup}\lambda
\left(\{1\}\times\mathbb{C}^{n}\right)}\\ \ & =
\overline{\underset{\lambda\in \mathbb{C}^{*}}{\bigcup}\lambda
\overline{G(1,x)}} \subset \overline{\widetilde{G}(1,x)}
\end{align*}
Hence $\mathbb{C}^{n+1}=\overline{\widetilde{G}(1, x)}$.
\end{proof}
\bigskip

\ We will use the following theorem, to prove Lemma ~\ref{L:1}:
\bigskip

\begin{thm}\label{T:5}$($\cite{aAh-M12},\ Theorem 1.1$)$   Let $\widetilde{G}$ be a finitely generated abelian semigroup of matrices
 of $M_{n+1}(\mathbb{C})$.
 If $J_{\widetilde{G}}(v)=\mathbb{C}^{n+1}$ for some $v\in V$ then
$\overline{\widetilde{G}(v)} = \mathbb{C}^{n+1}$.
\end{thm}
\bigskip

\begin{lem}\label{L:1} Let $\mathcal{G}$ be an abelian semigroup
of affine maps generated by $f_{1},\dots, f_{p}$ and $x\in
\mathbb{C}^{n}$. Then the following assertion are equivalent:\
\\ (i) $y\in J_{\mathcal{G}}(x)$.\ \\ (ii) $(1,y)\in
J_{\widetilde{G}}(1,x)$.
\end{lem}
\medskip

\begin{proof} $(i)\Longrightarrow (ii):$  Since $y\in J_{\mathcal{G}}(x)$, then there exists a sequence of vectors $(x_{m})_{m}$ with
$x_{m}\longrightarrow x$ and sequences of non-negative integers
$\{k^{(j)}_{m}:\ m\in \mathbb{N}\}$ for $j = 1,2,\dots, p$ with $$
(1.1)\ \ \ \ \ \ k^{(1)}_{m} + k^{(2)}_{m}+\dots+
k^{(p)}_{m}\longrightarrow+\infty$$ such that
$$f_{1}^{k^{(1)}_{m}}f_{2}^{k^{(2)}_{m}}\dots
f_{p}^{k^{(p)}_{m}}x_{m}\longrightarrow y.$$

Therefore, $(1,x_{m})\longrightarrow (1,x)$ and sequences of
non-negative integers $\{k^{(j)}_{m}:\ m\in \mathbb{N}\}$ for $j =
1,2,\dots, p$  such that
$$\Phi(f_{1})^{k^{(1)}_{m}}\Phi(f_{2})^{k^{(2)}_{m}}\dots
\Phi(f_{p})^{k^{(p)}_{m}}(1,x_{m})\longrightarrow (1,y).$$ Hence
$(1,y)\in J_{\widetilde{G}}(1,x)$.\ \\ $(ii)\Longrightarrow (i):$
Since $(1,y)\in J_{\mathcal{G}}(1,x)$, then there exists a
sequence of vectors $(\lambda_{m}, x_{m})_{m}\subset
\mathbb{C}\times \mathbb{C}^{n}$ with $(\lambda_{m},
x_{m})\longrightarrow (1,x)$ and sequences of non-negative
integers $\{k^{(j)}_{m}:\ m\in \mathbb{N}\}$ for $j = 1,2,\dots,
p$ with $$ (1.2)\ \ \ \ \ \ k^{(1)}_{m} + k^{(2)}_{m}+\dots+
k^{(p)}_{m}\longrightarrow+\infty$$ such that
$$\alpha_{1}^{k^{(1)}_{m}}\dots\alpha_{p}^{k^{(p)}_{m}}\Phi(f_{1})^{k^{(1)}_{m}}\Phi(f_{2})^{k^{(2)}_{m}}\dots
\Phi(f_{p})^{k^{(p)}_{m}}(\lambda_{m},x_{m})\longrightarrow
(1,y).$$

Denote by
$c_{m}=\alpha_{1}^{k^{(1)}_{m}}\dots\alpha_{p}^{k^{(p)}_{m}}$ and
$$g_{m}=\alpha_{1}^{k^{(1)}_{m}}\dots\alpha_{p}^{k^{(p)}_{m}}\Phi(f_{1})^{k^{(1)}_{m}}\Phi(f_{2})^{k^{(2)}_{m}}\dots
\Phi(f_{p})^{k^{(p)}_{m}}.$$ Then $\underset{m\to
+\infty}{lim}c_{m}\lambda_{m}=1$. As $\underset{m\to
+\infty}{lim}\lambda_{m}=1$, so $\underset{m\to
+\infty}{lim}c_{m}=1$. It follows that $\underset{m\to
+\infty}{lim}\frac{1}{c_{m}}g_{m}(1,x_{m})=(1,y)$. Hence
$\underset{m\to
+\infty}{lim}f_{1}^{k^{(1)}_{m}}f_{2}^{k^{(2)}_{m}}\dots
f_{p}^{k^{(p)}_{m}}(x_{m})=y$.
\end{proof}
\bigskip

\medskip

\begin{proof}[Proof of Theorem~\ref{t:1}] Suppose that
$J_{\mathcal{G}}(v)=\mathbb{C}^{n}$  with $v\in U$. By
Proposition~\ref{p:2}, we can assume that $G\subset
\mathcal{K}_{\eta,r}(\mathbb{C})$ and $P=I_{n+1}$. By
Lemma~\ref{L:1}, we have
$J_{\widetilde{G}}(1,v)=\mathbb{C}^{n+1}$. Then by
Theorem~\ref{T:5},
$\overline{\widetilde{G}(1,v)}=\mathbb{C}^{n+1}$. It follows by
Lemma~\ref{LL1L:9}, that
$\overline{\mathcal{G}(v)}=\mathbb{C}^{n}$.
\end{proof}
\bigskip

Now to prove Corollary ~\ref{c:1}, we need to use the following
proposition:

\begin{prop}\label{p:101}\cite{chm2} Let $G$ be an abelian sub-semigroup of $M_{n+1}(\mathbb{C})$ generated by
 $A_{1},\dots,A_{p}$, $p\geq 1$. Then $G$ is hypercyclic if and only if
 $J_{G}(x)= \mathbb{C}^{n+1}$ for every $x\in\mathbb{C}^{n+1}$.
\end{prop}
\medskip

\
\\
\\
{\it Proof of Corollary ~\ref{c:1}.} $(i)\Longrightarrow(ii)$
follows from Proposition ~\ref{p:101} and
$(ii)\Longrightarrow(iii)$ results from Theorem \ref{t:1}.
$(iii)\Longrightarrow(i):$ is trivial.
\medskip

\
\\
\\
{\it Proof of Corollary ~\ref{c:2}.} If $\mathcal{G}$ is not
hypercyclic then by Theorem \ref{t:1}, $J_{\mathcal{G}}(v)\neq
\mathbb{C}^{n}$  for any $v\in U'$, thus $U\cap
\mathcal{A}=\emptyset$ and therefore $\mathcal{A}
\subset\mathbb{C}^{n}\backslash U$. On the other hand,
$U=\varphi(U')$, so
\begin{align*}
  \mathbb{C}^{n}\backslash
U & =\mathbb{C}^{n}\backslash \varphi(U') \\
  \ & =\varphi(\mathbb{C}^{n}\backslash
U')\\ \ &
=\varphi\left(\underset{k=1}{\overset{r}{\bigcup}}L_{k}\right)\ \
\ \ \ \ \ (\mathrm{by}\ \ (1))\\ \ & =
\underset{k=1}{\overset{r}{\bigcup}}\varphi(L_{k})
\end{align*}

   with $L_{k} = \left\{x = [x_{1},\dots, x_{r}]^{T},\ x_{k}\in\{0\}\times\mathbb{C}^{n_{k}-1},\ x_{i}\in\mathbb{C}^{n_{i}},
\ \mathrm{if}\  i\neq k\right\}$. It follows that
$\mathcal{A}\subset\underset{k=1}{\overset{r}{\bigcup}}H_{k}$ with
$H_{k}=\varphi(L_{k})$ is an affine space with dimension $n-1$.
\qed

\bigskip

\bibliographystyle{amsplain}
\vskip 0,4 cm

\end{document}